\documentclass[12pt]{amsart}

\usepackage[
  margin=1in,
  includehead,          % count header height in the top margin
  headheight=14pt,      % bump this up if you see overlap/warnings
  headsep=10pt
]{geometry}
\usepackage[T1]{fontenc}
\usepackage[utf8]{inputenc}
\usepackage{lmodern}
\usepackage{microtype}
\usepackage{url}
\usepackage{graphicx, amssymb, amsmath, amsthm} % Required for inserting images
\usepackage{tikz-cd}
\usepackage{verbatim}
\usepackage{enumitem}
\usepackage{bm}
\usepackage{wrapfig}
\usepackage{needspace}   % preamble
\theoremstyle{plain}
\theoremstyle{plain}
\newtheorem{thm}{Theorem}[section]

\theoremstyle{definition}
\newtheorem{defn}[thm]{Definition} % definition numbers are dependent on theorem numbers
\newtheorem{exmp}[thm]{Example}
\newtheorem*{remark}{Remark}

\newtheorem{cor}[thm]{Corollary}

\theoremstyle{lemma}

\theoremstyle{proposition}
\newtheorem{prop}[thm]{Proposition}

\theoremstyle{conjecture}
\newtheorem{conj}[thm]{Conjecture}

\title{The Adjoint Polynomial of a Polyhedral cone is a Covolume Polynomial}
\author{Guanxi Li}
\date{September 2025}
\usepackage{natbib}
\bibpunct{[}{]}{,}{n}{}{,}
\begin{document}

\maketitle
\begin{abstract}
In this paper, we show that the adjoint polynomial of a polyhedral cone equals the multidegree polynomial of the toric ideal with multigrading, both given by the vertex rays. This fact implies a conjecture of Aluffi, to the effect that the adjoint polynomial of a polyhedral cone in the positive orthant is a covolume polynomial.
\end{abstract}
\section{Introduction}
The adjoint polynomial of a polyhedral cone was studied in \cite{Ark17} and \cite{kohn19} for its connections to positive geometry and canonical forms. In \cite{Alu24}, the adjoint polynomial is studied as an example of covolume polynomial, in the case when the polyhedral cone shares a face with a full dimensional simplicial cone enclosing it.
The definition of the adjoint polynomial of a polyhdedral cone is the following:
\begin{defn}
Let $C$ be a pointed polyhedral cone in $\mathbb{R}^{d+1}$ with triangulation $T$. Let $V(C), V(\sigma)$ denote the set of vertex rays of $C,\sigma$. The adjoint polynomial of $C$ is the polynomial in $t_1, \dots, t_{d+1}$ defined by 
$$adj_C(t) = \sum_{\sigma\in T}vol(\sigma)\prod_{v\in V(C)\backslash V(\sigma)}(v\cdot t),$$
where $vol(\sigma)$ is the normalized volume, that is, the $d!$ multiple of the ordinary volume.

\end{defn}
 Let $C\subseteq \mathbb{R}^{d+1}$ be a polyhedral cone, whose cross section with the hyperplane $\{x_0=1\}$ is a polytope $P\subseteq \mathbb{R}^d.$ 
In this situation, there is the following conjecture.
\begin{conj}[\cite{Alu24}]
If the polytope $P$ lies in the positive orthant, then the adjoint polynomial of $C$ is a covolume polynomial.
\end{conj}
We will recall the notion of covolume polynomial in Section 3.
In this paper, we prove the following:
\begin{thm}\label{main}
Let  $C$ be a polyhedral cone in as above. Let $\mathcal{A}$ be the point configuration corresponding to the generators of the cone. Furthermore, assume the points in $\mathcal{A}$ are $\mathbb{Z}^{d+1}$-points. Then, the adjoint polynomial of $C$ is the multidegree polynomial with respect to the $\mathcal{A}$-grading of the toric ideal $I_{\mathcal{A}}.$
\end{thm}
We will recall the definition of $I_{\mathcal{A}}$ in Section 2. The conjecture follows from \cite{cid24} Theorem 3.5,  which states that all multidegree polynomials with arbitrary $\mathbb{N}^{d+1}$-grading are covolume polynomials.
\vspace{5mm}
\section*{Acknowledgments}
The author thanks Paolo Aluffi for suggesting the problem and for countless insightful conversations. The author is supported by the Caltech Summer Undergraduate Research Fellowship (SURF) program and the Karen Roberts and Jim Sagawa SURF fellowship.
\vspace{5mm}
\section{Initial Toric Ideal and Triangulation}
Given a point configuration $\mathcal{A}:=\{v_1,v_2,\dots, v_n\}\subset \mathbb{Z}^{d+1}$,  
let $$\varphi: (\mathbb{C}^*)^{d+1}\rightarrow  (\mathbb{C}^*)^n,$$
$$p\mapsto (p^{v_1},p^{v_2},\dots, p^{v_n}), \text{   where } p^{v_i} \text{ denotes } p_{0}^{v_{i_{0}}}p_1^{v_{i_1} }\dots p_{d}^{v_{i_d}}.$$
The toric ideal $I_{\mathcal{A}}$ corresponding to the configuration is the defining ideal of the variety $im(\varphi)$.
Equivalently, the map above induces the ring homomorphism
$$\varphi^*: k[x_1^{\pm},x_2^{\pm},\dots,x_n^{\pm}]\rightarrow k[y_0^{\pm}, y_1^{\pm},y_2^{\pm},\dots y_d^{\pm}],$$
given by $x_i \mapsto y^{v_i}$, then $I_{\mathcal{A}} = ker(\varphi^*).$
More specifically, consider the $((d+1)\times n)$-matrix $A:= [v_1,v_2,\dots, v_n]$, then $I_{\mathcal{A}} = \langle x^{u}-x^v| Au=Av\rangle$ (\cite{MR96} Lemma 4.1).  

The polynomial ring $k[x_1, \dots, x_n]$ has a natural multigrading such that the ideal $I_{\mathcal{A}}$ is homogeneous with respect to that multigrading. This is the $\mathcal{A}$-grading defined as follows:
$deg(x_i) = v_i$.
In our situation, we are assuming the point configuration lives in $\{1\}\times\mathbb{Z}^d.$
\vspace{5mm}

A simplicial complex $\Delta$ is a set of subsets of $[n]:= \{1,2,\dots, n \},$ such that for $\sigma,\tau \subseteq [n],$ if $\sigma\supseteq\tau$ and $\sigma \in \Delta,$ then $\tau \in \Delta.$ If $\sigma, \tau\in \Delta,$ then $\sigma\cap \tau \in \Delta$. Elements of $\Delta$ are called faces of the simplicial complex.
Given a simplicial complex $\Delta$ on the vertex set $[n]$, there is a natural way to associate a square-free ideal $I_{\Delta}:= \langle x^{\tau}| \tau \notin \Delta \rangle$ to $\Delta$. Equivalently, one can also write the ideal in terms of the intersection of monomial prime ideals: $I_{\Delta} = \bigcap_{\sigma\in \Delta }\langle x_i|i\in \{1,2,\dots n\}\backslash \sigma\rangle$ (\cite{MR05} Theorem 1.7.)  This is the Stanley-Reisner ideal of the simplicial complex.

Conversely, consider an ideal $I\subseteq k[x_1,x_2,\dots, x_n]$ and a fixed term order given by the weight vector $\omega\in \mathbb{R}^n$(i.e, given monomials $x^{\alpha}, x^{\beta}$, then $x^{\alpha}\prec x^{\beta}$ if and only if $\omega\cdot\alpha < \omega\cdot\beta$). One obtains a simplicial complex on $\{1,2,\dots n\}$, called the initial complex $\Delta_{\omega}(I):=\{\sigma\ | \ \forall f\in I, \text{ the support of  }in_{\omega}(f) \text{ is not } \sigma\}.$ . Equivalently, the initial complex is  the simplicial complex whose Stanley-Reisner ideal is the radical of $in_{\omega}(I)$ (\cite{MR05} Chapter 8).

\vspace{5mm}
Given a point configuration $\mathcal{A}=\{v_1,v_2\dots, v_n\}$,  and a generic weight vector $\omega = (\omega_1, \omega_2, \dots \omega_n)$, the regular triangulation of $\mathcal{A}$, denoted by $\Delta_{\omega}$, is defined as follows:
Consider $\hat{\mathcal{A}}= \{(v_1,\omega_1), (v_2,\omega_2), \dots (v_n,\omega_n)\}\subseteq \mathbb{N}^{d+2}$, the polyhedral cone  $cone(\hat{\mathcal{A}})$, then consider the collection of all the faces of the cone whose supporting hyperplane lower bounds $cone(\hat{\mathcal{A}})$ with respect to the last coordinate. This collection forms a polyhedral complex, and the projection of this polyhedral complex with respect to the last coordinate is the triangulation $\Delta_{\omega}$. We can identify $\{v_1,v_2,\dots, v_n\}$ with $[n].$ Then $\Delta_{\omega}$ can be viewed as a simplicial complex on $[n].$
\newline
\par More precisely, for every generic choice of the weight vector $\omega$, the triangulation is defined as follows:
the subset $\{i_0,\dots, i_d\}$ is a face of $\Delta_{\omega}$ if there is a vector $c=(c_1,\dots,c_d)$ such that:
$$v_j\cdot c = \omega_j \text{  if } j \in \{i_0,\dots, i_d\} \text{ and}$$
$$v_j\cdot c < \omega_j \text{  if  } j \notin \{i_0,\dots, i_d\}.$$

\begin{exmp}\label{example}
%\leavevmode
%\Needspace{14\baselineskip} % if not enough space left on page, start a new one

%\begin{wrapfigure}[14]{l}{0.38\textwidth} % 14 ≈ number of lines the wrap spans
  %\vspace{-0.7\baselineskip}              % nudge upward to sit right under the title
 % \centering
  %\includegraphics[width=\linewidth]{polygon.png}
%\end{wrapfigure}

Consider the pentagon in $\mathbb{R}^2$ given by the points
$$v_1 = (1,0,1), v_2 = (1,1,1), v_3=(1,2,2), v_4 = (1,2,3), v_5 = (1,0,3)$$

Using Macaulay2 \cite{IC}, one is able to compute the ideal $I_{\mathcal{A}} = \langle x_3^2x_5-x_1x_4^2, x_2^2x_4-x_1x_3^2, x_2^2x_5-x_1^2x_4\rangle$.
\newline
\begin{center}
\vspace{-0.7\baselineskip}
\centering
  \includegraphics[width=0.36\textwidth]{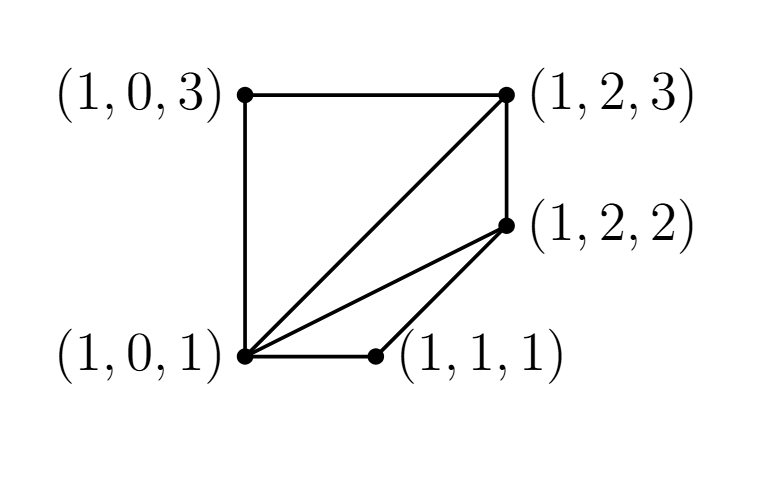}
  
\end{center}

\par Consider the weight vector $\omega= (0,1,0,0,1)$, we have the polytope 
\newline
$$P=conv((0,1,0),(1,1,1), (2,2,0),(2,3,0),(0,3,1)),$$
and denote each vertex $\tilde{v}_i$,  then the corresponding lower faces are $$conv(\tilde{v}_1,\tilde{v}_2, \tilde{v}_3), conv(\tilde{v}_1,\tilde{v}_3, \tilde{v}_4), conv(\tilde{v}_3,\tilde{v}_4, \tilde{v}_1).$$ Hence, the corresponding $2$-simplices in the triangulation are$\{1,2,3\}, \{1,3,4\}, \{1,4,5\}$.
\newline
\par The weight vector $\omega$ also induces a monomial ordering on $k[x_1,x_2,\dots, x_5]$. For example, $(0,0,2,0,1)\cdot (0,1,0,0,1)=1 \geq (1,0,0,2,0)\cdot (0,1,0,0,1)=0$, therefore $x_3^2x_5\succ x_1x_4^2.$ Similarly, $x_2^2x_4\succ x_1x_3^2, x_2^2x_5\succ x_1^2x_4.$ Using Macaulay2, one could also check that the generators above for $I_{\mathcal{A}}$ form a Grobner basis. Therefore $in_{\omega}(I_{\mathcal{A}}) = \langle x_3^2x_5, x_2^2x_4, x_2^2x_5\rangle.$

We see that the radical of the initial ideal is exactly the Stanley-Reisner ideal:
$$\sqrt{in_{\omega}(I_{\mathcal{A}})} =  \langle x_3x_5, x_2x_4, x_2x_5\rangle = \langle x_3,x_2\rangle \cap \langle x_5,x_2\rangle\cap \langle x_4,x_5\rangle = I_{\Delta_{\omega}}.$$
\end{exmp}
\vspace{10mm}
We have the following proposition relating the initial ideals of a toric ideal with respect to some monomial orderings and regular triangulations of the point configuration.

\begin{prop}[\cite{MR96} Theorem 8.3]
The regular triangulations of $\mathcal{A}$ are the initial complexes of the toric ideal $I_{\mathcal{A}}$. Let $\omega\in \mathbb{Z}^d$ be the weight vector. Then, the initial complex of $I_{\mathcal{A}}$ equals the triangulation of $\mathcal{A},$  $\Delta_{\omega}(I_{\mathcal{A}}) = \Delta_{\omega}$ . 
\end{prop}
\begin{cor}\label{min}
The minimal primes of $in_{\omega}(I_{\mathcal{A}})$ are of the form $\langle x_{i }| i\notin \sigma\rangle$ where $\sigma\in \Delta_{\omega}$ is a simplex with maximal dimension.
\end{cor}
\begin{proof}
Taking the Stanley-Reisner ideals corresponding to these simplicial complexes, we get $$\sqrt{in_{\omega}(I_{\mathcal{A}})} = \bigcap_{\sigma\in \Delta }\langle x_i|i\in [n]\backslash \sigma\rangle .$$
Note that if $\tau\prec \sigma,$ then $\langle  x_i| i\in [n]\backslash \sigma\rangle \subseteq \langle x_i|i\in [n]\backslash \tau  \rangle.$ 
Therefore, the minimal prime ideals of $in_{\omega}(I_{\mathcal{A}})$ are exactly ideals of the form, $$\langle x_i|i\in [n]\backslash \sigma, \text{where } \sigma \text{ is a maximal face of }\Delta \rangle.$$
\end{proof}
Hence, each maximal face in the subdivision corresponds to a minimal prime of the initial ideal. In the case when $\omega$ is generic, these maximal faces are $d-$simplices. The correspondence is captured in the following:
\begin{prop}[\cite{MR96} Theorem 8.8]\label{vol}
Let $\sigma$ be any $d$-simplex in the regular triangulation $\Delta_{\omega}(I_{\mathcal{A}}).$ The normalized volume of $\sigma $ equals the multiplicity of the prime ideal $\langle x_i:i\notin \sigma \rangle$ in $in_{\omega}(I_{\mathcal{A}}).$
\end{prop}
\begin{exmp}[Continuation of Example \ref{example}]
It is easy to compute the primary decomposition of $in_{\omega}(I_{\mathcal{A}}) = \langle x_3^2x_5, x_2^2x_4, x_2^2x_5\rangle,$ which is $in_{\omega}(I_{\mathcal{A}})= \langle x_2^2, x_3^2\rangle \cap  \langle x_2^2, x_5\rangle \cap \langle x_4, x_5 \rangle$.  Then the multiplicities of the minimal primes are $m_{in_{\omega}(I_{\mathcal{A}})}(\langle x_2^2,x_3^2\rangle)= dim_k(k[x_2,x_3]/\langle x_2^2,x_3^2\rangle) =4$, $m_{in_{\omega}(I_{\mathcal{A}})}(\langle x_2^2,x_5\rangle) = 2,$ $m_{in_{\omega}(I_{\mathcal{A}})}(\langle x_4,x_5\rangle)=1,$ which are $2!$ multiples of the volumes of the simplices $\{1,2,3\},\{3,4,1\},\{4,5,1\}.$
\end{exmp}
\vspace{10mm}
\newpage
\section{Multidegree Polynomial}

\begin{defn}[\cite{MR05}, Definition 8.43]
Let $S= k[x_1,x_2,\dots,x_n]$ be the polynomial ring, with the multigrading $deg(x_i)=a_i$. Let $\mathcal{C}:M\mapsto \mathcal{C}(M;t)$ be the function from a finitely generated $S$-module to $\mathbb{Z}[t_1,t_2,\dots,t_d]$, satisfying the following two properties:
\begin{enumerate}
\item(Additivity) $$\mathcal{C}(M;t) = \sum_{k=1}^r mult_{\mathfrak{p}_k}\mathcal{C}(S/\mathfrak{p}_k;t),$$
where $\mathfrak{p}_1,\dots, \mathfrak{p}_r$ are the maximal associated primes of $M.$
\item (Degenerativity)
\newline
if $M := \mathcal{F}/K$ is a graded free presentation and $in(M):=\mathcal{F}/in(K)$ for some term order,
$$\mathcal{C}(M;t) = \mathcal{C}(in(M);t).$$
\end{enumerate}
\end{defn}

\begin{prop}[\cite{MR05} Theorem 8.44]\label{lin}
There is exactly one such function satisfying
$$\mathcal{C}(S/\langle x_{i_1},\dots x_{i_r}\rangle;t) = \langle a_{i_1},t\rangle\cdots\langle a_{i_r},t \rangle$$
for all prime monomial ideals $\langle x_{i_1},\dots,x_{i_r}\rangle$, where $\langle a,t \rangle$ is the usual dot product.
\end{prop}
Consider the polynomial ring $k[x_1,x_2,\dots, x_n]$ with a fixed multigrading $deg(x_i)=a_i$. The multidegree polynomial with respect to this grading is defined as the unique function $\mathcal{C}$ that satisfies the condition in the proposition above.

\begin{remark}
Multidegree polynomial can also be defined as the $codim(M)$ terms of $\mathcal{K}(M;1-t),$ the $\mathcal{K}$-polynomial with respect to a fixed $(a_{i_1},\dots, a_{i_r})$-multigrading (\cite{MR05} Definition 8.45). In this paper, we will not make use of this fact. 
\end{remark}

\vspace{5mm}
\begin{proof}[Proof of  Theorem \ref{main}]
Expand the multidegree polynomial by first passing to the initial ideal, then using additivity:
$$\mathcal{C}(S/I_{\mathcal{A}};t) = \mathcal{C}(S/in_{\omega}(I_{\mathcal{A}})) =  \sum_{k=1}^r mult_{S/in_{\omega}(I_{\mathcal{A}})}(\mathfrak{p}_k) \mathcal{C}(S/\mathfrak{p}_k;t),$$
where $\mathfrak{p}_i$ are the maximal associated primes of $S/in_{\omega}(I_{\mathcal{A}})$, that are the minimal prime ideals of $in_{\omega}(I_{\mathcal{A}})$. When $\omega$ is generic, $in_{\omega}(I_{\mathcal{A}})$ is a monomial ideal. Then, by Corollary \ref{min}, the minimal primes are monomial prime ideals of the form $$\langle x_i|i\in [n] \backslash \sigma, \text{where } \sigma \text{ is a maximal simplex in }\Delta \rangle.$$
Then,  by Proposition \ref{vol}, $mult_{S/in_{\omega}(I_{\mathcal{A}})}(\mathfrak{p}_k) = vol(\sigma)$ the normalized volume and by Proposition \ref{lin}, $$\mathcal{C}(S/\mathfrak{p}_k;t) =  \mathcal{C}(S/\langle x_i| i\in \sigma \rangle)= \prod_{v\notin V}(v\cdot x).$$Then $\mathcal{C}(S/I_{\mathcal{A}};t) =adj_C(x).$
\end{proof}
\vspace{5mm}
The notion of the covolume polynomial was introduced in \cite{Alu24}.
\newline
Let $X\subseteq \mathbb{P}:= \mathbb{P}^{n_1}\times\dots \mathbb{P}^{n_l} $ be an irreducible subvariety. Then in $A_{*}(\mathbb{P}),$
$$[X] = \sum_{\substack{i_1,\cdots,i_l \\ i_1+\dots+i_n=codim(X)}}a_{i_1,\dots i_{l}}h_1^{i_1}\cdots h_l^{i_l}\cap [\mathbb{P}].$$
For variables $t_1,\dots, t_l,$ one can associate with $[X]$ the polynomial $$P_{[X]}(t_1,\dots, t_l) =\sum_{\substack{i_1,\cdots,i_l \\ i_1+\dots+i_n=codim(X)}}a_{i_1,\dots i_{l}}t_1^{i_1}\cdots t_l^{i_l} .$$
\begin{defn}
A polynomial is called a covolume polynomial if it is the limit of polynomials of the form $cP_{[X]}(t)$ where $c$ is a positive real number.
\end{defn}
 THe covolume polynomial is then subsequently being studied in \cite{huh24} for it connection to the volume polynomial and algebraic matroids. The following example gives more insight into the name of the multidegree polynomial.

\begin{exmp}
Let $X\subseteq \mathbb{P}^{n_1}\times \mathbb{P}^{n_2} \dots \times \mathbb{P}^{n_l}$ be defined by the multihomogeneous ideal $I\subseteq R:=k[x_1,\dots, x_n]$ with respect to standard grading, i.e., $deg(x_i) = e_i\in \mathbb{Z}^n$. Then, 
$$P_{[X]}(t) = \mathcal{C}(R/I;t) = \sum_{\substack{0\leq i _j\leq n_j\\ i_1+i_2\dots+i_l=codim(X)}}deg_{\mathbb{P}}^{\mathbf{i}}(X)t_1^{i_1}\cdots t_l^{i_l},$$
where $deg_{\mathbb{P}}^{\mathbf{i}}(X)$ is the multidegree of $X$ of type $\mathbf{i}:=(i_1,i_2,\dots, i_l).$

\end{exmp}
\begin{prop}[\cite{cid24} Theorem 3.5]
Consider $R = k[x_1,\dots, x_n]$ with $\mathbb{N}^{d}$-grading, and $I$ a prime ideal, homogeneous with respect to the multigrading. Then, $\mathcal{C}(R/I;t)$ is a covolume polynomial.
\end{prop}

\begin{cor}
Let $C\subseteq \mathbb{R}^{d+1}$ be a polyhedral cone, whose cross section with the hyperplane $\{x_0=1\}$ is  a polytope $P\subseteq \mathbb{R}^d.$ Furthermore, assume all vertices of the polytope lie in the positive orthant in the cross section of the polyhedral cone. Then the adjoint polynomial $adj_C$ is a covolume polynomial. 
\end{cor}
\begin{proof}
By Theorem \ref{main} and \cite{cid24} Theorem 3.5, when the vertices of the polytope are $\mathbb{N}^d$-points, $adj_{C}$ is a covolume polynomial.
If the vertices of $P$ are $\mathbb{Q}^d$-points that lie in the positive orthant, then one can scale the coordinates of the vertex rays of the polyhedral cone by a common multiple of the denominator in each coordinate. The resulting point configuration is integral. Then, the adjoint polynomial differs by a positive scalar multiple from the multidegree of a toric ideal with nonnegative multigrading. Hence, the adjoint polynomial is still a covolume polynomial.
Now, if the vertices are not rational, consider a sequence of polyhedral cones whose vertex rays are rational, $\{C_i\}_{i=0}^{\infty}$, such that $\lim_{i\rightarrow\infty}C_i = C$. By \cite{kohn19} Corollary 1.8, $\lim_{i\rightarrow\infty}adj_{C_i}(x) = adj_C(x)$. Therefore, $adj_{C}$ is also a covolume polynomial.
\end{proof}

\nocite{*}

\bibliographystyle{alpha}

\bibliography{references.bib}

\end{document}